\documentclass[11pt]{article} 
\usepackage[bmargin=30mm,tmargin=30mm,rmargin=40mm,lmargin=40mm]{geometry}
\usepackage{amsmath,amsthm,amsfonts,graphicx,enumerate}
\usepackage[english]{babel}
\usepackage[numbers,sort&compress]{natbib}
\usepackage[colorlinks=true,citecolor=black,linkcolor=black,urlcolor=black]{hyperref}
\newcommand{\doi}[1]{doi:\,\href{http://dx.doi.org/#1}{#1}}
\newcommand{\titel}{Thoughts on Barnette's Conjecture}

\theoremstyle{plain} 
	\newtheorem{satz}{Satz}[] 
	\newtheorem{theorem}[satz]{Theorem}
	\newtheorem{lemma}[satz]{Lemma}

	\newtheorem*{conjecture*}{Conjecture} 
	\newtheorem{corollary}[satz]{Corollary}

\begin{document}
	\title{\bf\titel\footnote{Research supported by an Australia--Germany Go8--DAAD collaboration grant.}}
	\author{Helmut Alt\footnotemark[2] \and Michael S. Payne\footnotemark[3] \and Jens M. Schmidt\footnotemark[4] \and David R. Wood\footnotemark[5]}
	
	\maketitle
	
	\footnotetext[2]{Institute of Computer Science, Freie Universit\"at Berlin, Germany (\texttt{alt@mi.fu-berlin.de}).}

	\footnotetext[3]{Department of  Mathematics and Statistics, The University of Melbourne, 
  Australia   (\texttt{m.payne3@pgrad.unimelb.edu.au}). Research supported by an Australian Postgraduate Award. }

	\footnotetext[4]{Department of Algorithms and Complexity, Max Planck Institute for Informatics, Saarbr\"ucken, 
Germany (\texttt{jens.schmidt@mpi-inf.mpg.de}).}

	\footnotetext[5]{School  of  Mathematical Sciences, Monash University, Melbourne, Australia (\texttt{david.wood@monash.edu}). Research supported by the Australian Research Council.}

\begin{abstract}
We prove a new sufficient condition for a cubic 3-connected planar graph to be Hamiltonian. This condition is most easily described as a property of the dual graph. Let $G$ be a planar triangulation. Then the dual $G^*$ is a cubic 3-connected planar graph, and $G^*$ is bipartite if and only if $G$ is Eulerian. We prove that if the vertices of $G$ are (improperly) coloured blue and red, such that the blue vertices cover the faces of $G$, there is no blue cycle, and every red cycle contains a vertex of degree at most 4, then $G^*$ is Hamiltonian.

This result implies the following special case of Barnette's Conjecture: if $G$ is an Eulerian planar triangulation, whose vertices are properly coloured blue, red and green, such that every red-green cycle contains a vertex of degree 4, then $G^*$ is Hamiltonian. Our final result highlights the limitations of using a proper colouring of $G$ as a starting point for proving Barnette's Conjecture. We also explain related results on Barnette's Conjecture that were obtained by Kelmans and for which detailed self-contained proofs have not been published.
\end{abstract}


\section{Introduction}

The study of Hamiltonian cycles in cubic planar graphs has a rich history, originally motivated by Tait's conjecture that every cubic 3-connected planar graph is Hamiltonian (which implies the 4-colour theorem). Tutte~\cite{Tutte} disproved Tait's conjecture, which led to the following conjecture of Barnette:

\begin{conjecture*}[Barnette~\cite{Barnette1969a}]
\label{BarnettesConjecture}
Every cubic $3$-connected planar bipartite graph is Hamiltonian.
\end{conjecture*}

This paper proves a new sufficient condition for a  cubic 3-connected planar graph to be Hamiltonian. This condition is most easily described as a property of the dual graph. The dual of a cubic 3-connected planar graph $G$ is a planar triangulation $G^*$, and $G$ is bipartite if and only if $G^*$ is Eulerian (that is, every vertex has even degree). The following is our main result. 

\begin{theorem}
\label{Main}
Let $G$ be a planar triangulation, whose vertices are coloured blue and red, such that the blue vertices hit every face of $G$, there is no blue cycle, and every red cycle contains a vertex of degree at most 4. Then $G^*$ is Hamiltonian.
\end{theorem}

Note that every Eulerian triangulation has a unique proper $3$-colouring. 
Theorem~\ref{Main} implies the following corollary for Eulerian triangulations, which via duality can be thought of as a particular case in which Barnette's Conjecture holds. This corollary is also implied by a recent result of Florek~\cite{Florek12}.

\begin{corollary}
\label{ColouredCorollary}
Let $G$ be an Eulerian planar triangulation, whose vertices are properly coloured blue, red and green, such that every red-green cycle contains a vertex of degree 4. Then $G^*$ is Hamiltonian.
\end{corollary}

\begin{proof}
Apply Theorem~\ref{Main} with the red and green vertices all coloured red. There is no blue cycle since the blue vertices are an independent set.
By assumption, every red cycle contains a vertex of degree 4. 
\end{proof}

It is interesting to note that Theorem~\ref{Main} also implies the following corollary of Florek~\cite[Corollary 2.1]{Florek12}.

\begin{corollary}[Florek]
Let $G$ be an Eulerian planar triangulation, whose vertices are properly coloured blue, red and green. Let $X$ and $Y$ partition the vertices of degree at least $6$ such that all such red vertices are in $X$ and all such blue vertices are in $Y$. If the induced graphs $G[X]$ and $G[Y]$ are acyclic, then $G^*$ is Hamiltonian.
\end{corollary}

\begin{proof}
Let $R$ be the red vertices of $G$. 
In order to apply Theorem~\ref{Main}, initially set $V_1 = X \cup R$ and $V_2 = V \setminus V_1$.
Thus $V_1$ hits all faces and every cycle in $G[V_2]$ contains a vertex of degree $4$ (since $G[Y]$ is acyclic).
It is also required that $G[V_1]$ is acyclic.  
Any cycle in $G[V_1]$ contains a red vertex $v$ of degree $4$.
The vertex $v$ has two green and two blue neighbours and the blue ones are in $V_2$.
So the cycle goes through the two green neighbours. These green neighbours touch all faces adjacent to $v$, so we may move $v$ from $V_1$ to $V_2$. 
Do this until there are no cycles in $G[V_1]$, then apply Theorem~\ref{Main}.
\end{proof}

All our results are based on the following definition. Let $G$ be a planar triangulation. A subgraph $H$ of $G$ \emph{permeates} $G$ if $H$ is induced and $H$ hits every face (that is, each face of $G$ is incident to some vertex in $H$). It is well known that  $G^*$ is Hamiltonian if and only if $G$  contains a permeating  subtree, and that the complement of a permeating subtree is another permeating subtree (see Section~\ref{sec:PermeatingTree}). 
Our final result shows that any approach that, like Corollary~\ref{ColouredCorollary}, constructs a permeating subtree with all or no vertices from a colour class of the proper $3$-colouring of an Eulerian triangulation, is insufficient to prove Barnette's Conjecture. 

\begin{theorem}
\label{Construction}
For every integer $k$ there is a properly $3$-coloured Eulerian planar triangulation $G$ 
such that every permeating subtree of $G$ contains at least $k$ vertices from each colour class, and excludes at least $k$ vertices from each colour class.
\end{theorem}

Extensive surveys of the many results relating to Barnette's Conjecture~\cite{Lu,Florek,Florek12,Krooss,Chia,ChiaOng,Barnette1969a,Borowiecki,Holton,Feder,Plummer,Hertel,laTorre,Kelmans,HoltonAldred00,Aldred00,Aldred99,HoltonAldred} can be found in~\cite{Hertel,laTorre}.
One important result due to Kelmans~\cite{Kelmans} establishes equivalence between Barnette's Conjecture and several apparently different statements.
For example, to prove the dual version of Barnette's Conjecture it is sufficient to consider only triangulations without separating triangles. 
Since they are of significant interest but the original publications contained very few details, in Section~\ref{sec:kelmans} we take the opportunity to explain Kelmans' proofs.

\section{A Useful Lemma}
\label{sec:PermeatingTree}

The following  well-known lemma characterises when the dual of a planar triangulation is Hamiltonian \cite{Stein,Skupien,Florek,Florek12}. We include a proof for completeness. See Figure~\ref{bstet} for an example. 
 
\begin{lemma}
\label{lem:PermeatingTree}
The following are equivalent for a planar triangulation $G$:
\begin{enumerate}
\item $G^*$ is Hamiltonian,
\item $G$ contains a permeating subtree,
\item $G$ contains two disjoint permeating subtrees that partition $V(G)$. 
\end{enumerate}
\end{lemma}

\begin{proof}
(2) $\Longrightarrow$ (3): Let $T$ be a permeating subtree, and let $T'$ be the subgraph of $G$ induced by the vertices not in $T$. 
$T'$ is permeating, otherwise $T$ contains a whole facial cycle.
If $T'$ contains a cycle $C$, then $T$ lies either inside or outside $C$ (since $T$ is connected).
The other side of $C$ contains a face of $G$ that $T$ does not hit.
Hence $T'$ is acyclic, and $T\cup T'$ is a forest.

It remains to show that $T'$ is connected.
Suppose $G$ has $n$ vertices, $m$ edges and $f$ faces.
For each triangle of $G$, $T \cup T'$ contains one of its three edges, so the number of edges is $f/2$.
Since $G$ is a triangulation, $m=3f/2$.
Combining this with Euler's formula yields $f/2 = n-2$.
An $n$-vertex forest with $n-2$ edges has two components, 
so $T'$ is connected.

(3) $\Longrightarrow$ (2): Trivial.

(3) $\Longrightarrow$ (1): Contracting each of the two trees to a single vertex leaves all faces intact, and the resulting multigraph has two vertices and no loops. The dual of this graph is a cycle, which corresponds to a Hamiltonian cycle in $G^*$.

(1) $\Longrightarrow$ (3): Let $C$ be a Hamiltonian cycle in $G^*$. 
The cycle $C$ determines a closed Jordan curve $C'$ in the plane that avoids all the vertices of $G$, and only crosses the edges of $G$ that are dual to the edges of $C$. 
Let $T_1$ be the subgraph of $G$ induced by  the vertices inside of $C'$, and $T_2$ the subgraph of $G$ induced by the vertices outside of $C'$. 
Clearly $T_1$ and $T_2$ partition $V(G)$.
If $T_1$ or $T_2$ contains a cycle, then that cycle would contain a face of $G$ not met by $C$, which contradicts the Hamiltonicity of $C$. 
Hence $T_1 \cup T_2$ is a forest. 
In particular, no facial cycle is contained in $T_1$ or $T_2$, so both $T_1$ and $T_2$ are permeating.

To show that $T_1$ and $T_2$ are trees, suppose $G$ has $n$ vertices, $m$ edges and $f$ faces. 
Then the number of edges in $C$ is $f$, and the edges between $T_1$ and $T_2$ are precisely the $f$ edges dual to the edges of $C$. 
Hence by Euler's formula $T_1$ and $T_2$ contain a total of $m-f=n-2$ edges. 
An $n$-vertex forest with $n-2$ edges has two components, so $T_1$ and $T_2$ are trees.
\end{proof}

\begin{figure}[!ht]
\begin{center}\includegraphics{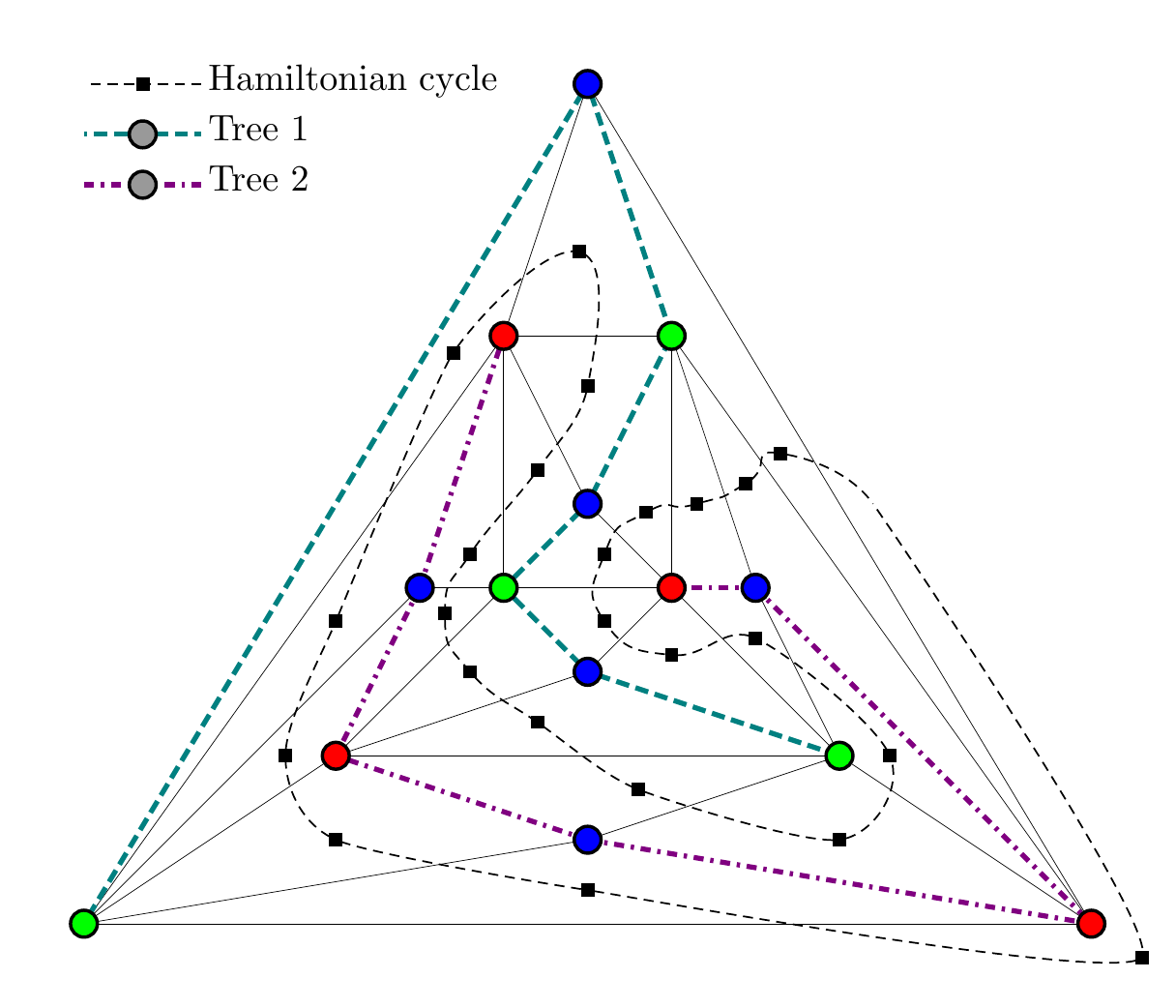} \end{center}   
\vspace*{-5ex}
\caption{A planar triangulation $G$ with two disjoint permeating subtrees and the corresponding Hamiltonian cycle  in $G^*$.}
\label{bstet}
\end{figure}

Note that the fact that $G$ is a triangulation is only used in (2)$\Longrightarrow$(3). (1)$\iff$(3) holds for all $2$-connected planar graphs. Also note that the permeating property could be omitted in (3) since two disjoint trees that cover all vertices must necessarily be permeating.

\section{Proof of Theorem~\ref{Main}}

Let $B$ be the set of blue vertices in  $G$. Let $R$  be the subgraph of $G$ induced by the red vertices. An edge of $R$ is \emph{short} if at least one of its endpoints has degree at most 4 in $G$, otherwise the edge is \emph{long}. Let $H$ be the planar dual of $R$. Note that $H$ may have loops or parallel edges. Each edge of $H$ is \emph{short} or \emph{long} depending on whether the dual edge in $R$ is short or long. 

Since every planar dual is connected, $H$ is connected.  Let $H'$ be the spanning subgraph of $H$ consisting of the short edges. We now prove that $H'$ is connected. Suppose, on the contrary, that $H'$ is disconnected. Since $H'$ is obtained from the connected graph $H$ by deleting the long edges,  some set $C$ of long edges  form a minimal edge cut in $H$. Let $C^*$ be the set of edges of $R$ that are dual to the edges in $C$. By planar duality, $C^*$ is a cycle in $R$. Every edge in $C^*$ is long. Thus $G$ contains a red cycle all of whose vertices have degree at least 5 in $G$. This contradiction proves that $H'$ is connected. Let $T$ be a spanning tree of $H'$. 

Construct a set $S$ of red vertices as follows. Consider each edge $vw$ in $T$ in turn. Let $xy$ be the short edge of $R$ that is dual with $vw$. At least one of $x$ and $y$, say $x$, has degree at most $4$ in $G$. Add $x$ to $S$. 

We now prove  that $G[B\cup S]$ is a permeating subtree of $G$. Since $B$  hits every face, $B\cup S$ hits every face. 
Since no face of $G$ is all red, no face of $R$ is a face of $G$. Hence there is a bijection between  the faces of $R$ and the connected components of $G[B]$, and thus also with the vertices of $T$. 
Associated with each edge of $T$ is a 2-edge path in $G[B\cup S]$ that joins two connected components of $G[B]$ via a red vertex in $S$. Hence $G[B\cup S]$ is connected. To conclude that $G[B\cup S]$ is a permeating subtree, we  now show that $B\cup S$ induces no other edges. 

First suppose that there is a vertex $x$ in $S$ with degree 3 in $G$. Say $xy$ was the short edge in $R$  when $x$ was added to $S$. Let $vxy$ and $wxy$ be the faces of $G$ incident to $xy$. Thus $v$ and $w$ are adjacent blue vertices, and $vxw$ is a face of $G$ (since $\deg_G(x)=3$). Hence the edge of $H$ dual with $xy$ is a loop, and is in no spanning tree of $H$. Therefore every vertex in $S$ has degree 4 in $G$.

Consider a vertex $x$ in $S$. So $x$ is red and has degree 4. If $x$ has four blue neighbours, then they form a blue cycle. If $x$ has three blue neighbours and $y$ is the red neighbour of $x$, then the blue neighbours induce a path in $G$, implying that the edge of $H$ dual to $xy$ is a loop, in which case $x$ is not added to $S$. If $x$ has three red neighbours, then some face incident to $x$ is not hit by the blue vertices. Hence $x$ has exactly two blue neighbours. Moreover, the blue neighbours of $x$ are not consecutive in the cyclic order around $x$, as otherwise some face incident to $x$ is not hit by the blue vertices. By construction, these two blue neighbours of $x$ are in distinct components of $G[B]$. 




Suppose that $S$ contains two adjacent vertices  $x$ and $x'$. Let $e$ be the short edge in $R$ incident to $x$ when $x$ was added to $S$. Let $e'$ be the short edge in $R$ incident to $x'$ when $x'$ was added to $S$. Let $v$ and $w$ be the two blue neighbours of $x$. Then $v,x',w$ are consecutive in the cyclic ordering of neighbours of $x$, which implies that $v$ and $w$ are also the two blue neighbours of $x'$. Thus the dual edge to both $e$ and $e'$ is $vw$. Hence $T$ contains a 2-cycle. This contradiction proves that no two vertices in $S$ are adjacent. 
Therefore $G[B\cup S]$ is a permeating subtree. Theorem~\ref{Main} follows by Lemma~\ref{lem:PermeatingTree}.

\section{Proof of Theorem~\ref{Construction}}

The proof of this theorem makes use of the special graph $H$ shown $3$-coloured in Figure~\ref{H} which can be considered as the ``barycentric subdivision'' of the tetrahedron. It has the following property:

\begin{figure}[!ht]
\begin{center}\includegraphics{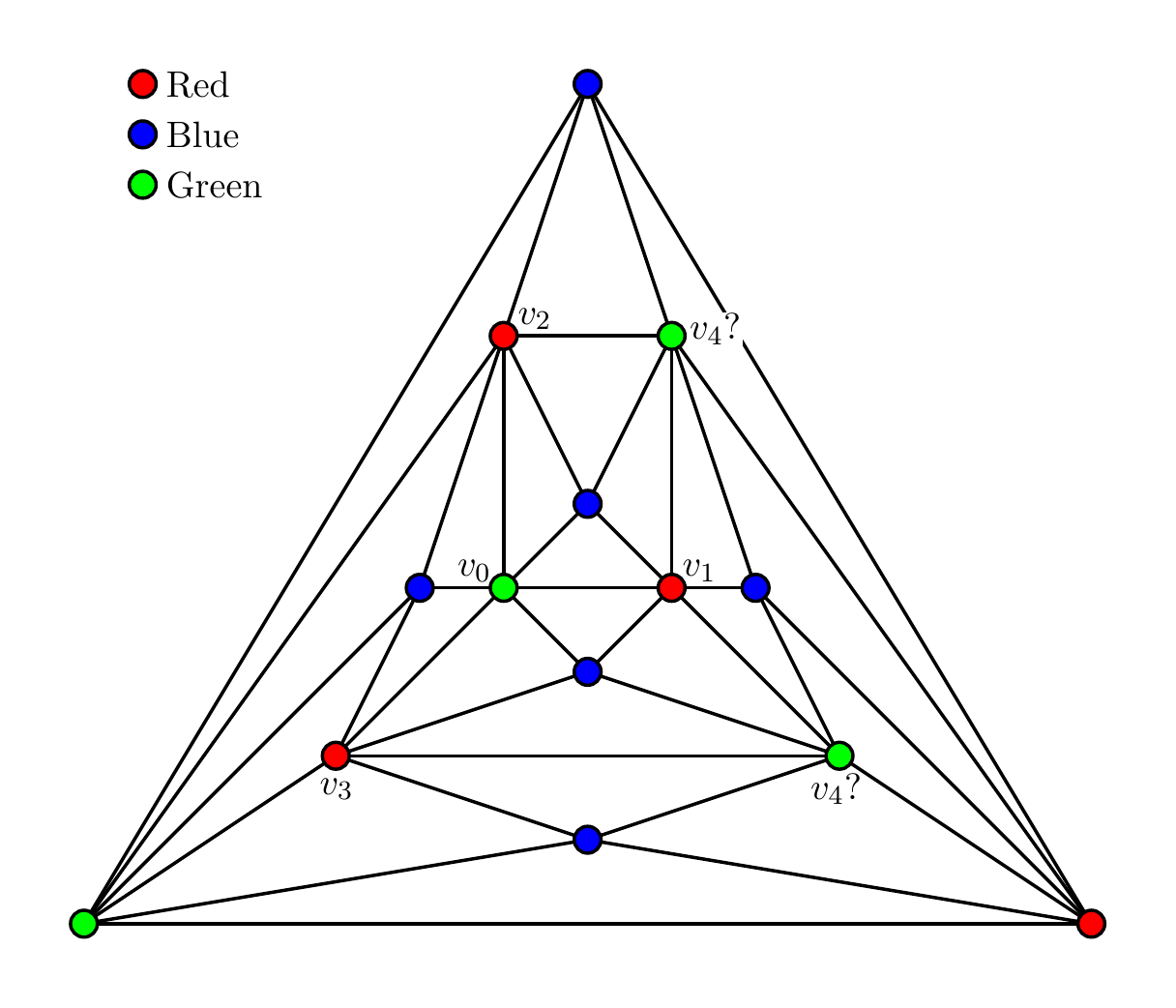} \end{center}   
\vspace*{-5ex}
\caption{The planar triangulation $H$.} 
\label{H}
\end{figure}

\begin{lemma}
\label{deg4}
Every permeating subtree of $H$ has at least one vertex of degree 4.
\end{lemma}

\begin{proof}
Observe that one  colour class (say \emph{blue}) contains exactly the vertices of degree 4, whereas in the other two colour classes (say \emph{green} and \emph{red}) all vertices have degree 6. Assume that there exists a permeating subtree $T$ of $H$ containing no blue vertex.

We first show that $T$ is a path. On the contrary, assume that there is a vertex $v_0$ with three neighbours in $T$. Without loss of generality, $v_0$ is green. Then its three red neighbours $v_1, v_2,v_3$ are in $T$. These four vertices do not hit all the faces of $H$. So, without loss of generality, $v_1$ has a green neighbour  $v_4 \ne v_0$ in $T$, as shown in Figure~\ref{H}. There are two possibilities for $v_4$, one adjacent to $v_2$,  the other  adjacent to $v_3$. So adding $v_4$ creates a cycle, contradicting that $T$ is an induced tree. Hence, $T$ is a path.


Each vertex of the path $T$ has six incident faces, two of which overlap with the faces incident to the previous vertex in $T$. Consequently, the number of faces incident to the path is congruent to $2 \mod 4$, which cannot equal the total number of faces ($24$). This  contradiction  proves Lemma~\ref{deg4}.
\end{proof}

Let $H_1,\dots,H_{3k}$ be copies of $H$ such that the degree-4 vertices in $H_i$ are coloured $i \bmod 3$. For $1\leq i\leq 3k$, let $g_i$ be the outer face of $H_i$, and let $f_i$ be a face of $H_i$ vertex-disjoint from $g_i$. Construct $G_1,\dots,G_{3k}$ recursively as follows.
Let $G_1 := H_1$. Then for $2\leq i\leq 3k$, construct $G_i$ by pasting $G_{i-1}$ and $H_i$ on faces $f_{i-1}$ and $g_i$, as illustrated in Figure~\ref{glue}. More precisely, each vertex in face $f_{i-1}$ of $G_{i-1}$ is identified with the vertex of the same colour in face
$g_i$ of $H_i$. Note that after the gluing, $f_i$ is a face of $G_i$ (so the construction makes sense). Observe that $G_i$ is a $3$-coloured Eulerian triangulation with $11i+3$ vertices. In particular, $G_{3k}$ is a $3$-coloured Eulerian triangulation with $33k+3$ vertices. 

\begin{figure}[!ht]
\begin{center}
\includegraphics[width=8cm]{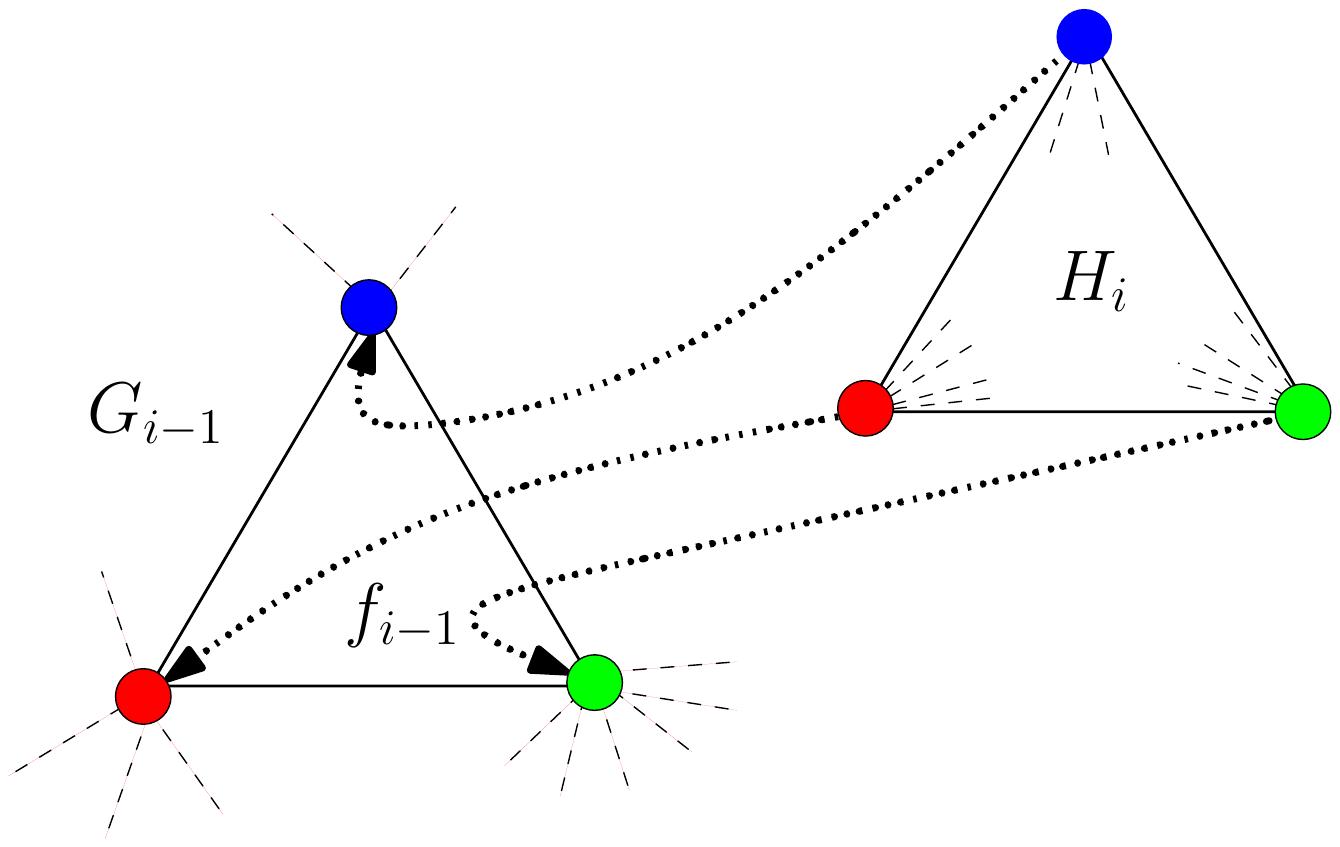}  
\end{center} 
\caption{$G_i$ is constructed by pasting a copy of $H$ into face $f_{i-1}$ of $G_{i-1}$.} 
\label{glue}
\end{figure}

We now make the following important observation:

\begin{lemma}
\label{inductree}
For $2\leq i\leq 3k$, if $T$ is a permeating subtree of $G_{i}$, then $T\cap H_i$ is  a permeating subtree of $H_i$ and $T\cap G_{i-1}$ is a permeating subtree of $G_{i-1}$. 
\end{lemma}

\begin{proof}
$T$ includes at least one vertex in $g_i$ since $g_i$ separates the two disjoint faces $g_1$ and $f_i$. Since $T$ is an induced tree, $T$ includes at most two vertices in $g_i$, and if $T$ includes two vertices in $g_i$ then they are adjacent in $G_i$ and thus in $T$. Hence $T\cap H_i$ is a (connected) subtree of $H_i$. Every face of $H_i$ except $g_i$ is a face of $G_i$ and is therefore hit by $T$. As observed above, at least one vertex of $g_i$ is in $T$. Thus, $T\cap H_i$ is  a permeating subtree of $H_i$. By an identical argument, $T\cap G_{i-1}$ is a permeating subtree of $G_{i-1}$. 
\end{proof}

Let $T$ be a permeating subtree of $G_{3k}$. Then, by induction applying Lemma~\ref{inductree} at each step, 
$T\cap H_i$ is a permeating subtree of $H_i$ for $1\leq i\leq 3k$. By Lemma~\ref{deg4},  $T$ includes at least one degree-4 vertex in each $H_i$.
Thus, for  $j \in \{0,1,2\}$, $T$ contains at least one vertex coloured $j$ in each $H_i$ such that $i \equiv j \pmod 3$.
Since $f_i$ and $g_i$ are vertex-disjoint, if $H_i$ and $H_{i'}$ have a vertex in common then $|i-i'| \leq 1$.
Thus these vertices coloured $j$ are distinct. Hence, $T$ contains at least $k$ vertices coloured $j$.
This completes the proof of Theorem~\ref{Construction}. \qed

\bigskip

A \emph{separating triangle} is a $3$-cycle whose deletion disconnects the graph.
In the following section we will see that the most important graphs for Barnette's conjecture are those without separating triangles.
Although the example just constructed has many separating triangles, a similar but somewhat more complicated construction can be used to obtain examples without separating triangles. We omit the details.

\section{Kelmans' Equivalences}\label{sec:kelmans}

In this section we explain some important results obtained by Kelmans~\cite{Kelmans}, for which detailed self-contained proofs have not been published. We hope that this will be of help to those who investigate Barnette's Conjecture in the future. Our aim here is to clearly present the main ideas of the proofs. Some details are still left to the reader to verify.

A planar $3$-connected bipartite graph will be called simply a \emph{Barnette graph}.
Let $x$ and $y$ be edges in the same face of a Barnette graph $G$. Then $G$ is \emph{$(x^+y^-)$-Hamiltonian} if it has a Hamiltonian cycle containing $x$ and not $y$. Similarly, $G$ is \emph{$(x^+y^+)$-Hamiltonian} if it has a Hamiltonian cycle containing $x$ and $y$.

A graph is \emph{cyclically-$k$-edge-connected} if it has no edge cut of size $k-1$ such that both sides of the cut contain a cycle. (In the dual, cyclic edge cuts correspond to separating cycles.)
All Barnette graphs are cyclically-$3$-edge-connected, but also all have a face of size $4$ so they cannot be cyclically-$5$-edge connected. (This may need proof for small examples).

Given two edges $x$ and $y$ in the same face of a Barnette graph $G$, the \emph{four-pole} $G'$ is the graph formed by cutting $x$ and $y$ and adding degree one vertices to the new ends, as illustrated in the left side of Figure~\ref{fig:Kelmans1and2}. 
The new vertices are called $x_1,x_2,y_1$ and $y_2$.
We will connect four-poles together in various ways to build up larger Barnette graphs.

\begin{figure}[!ht]
\begin{center}
\includegraphics[width=5cm]{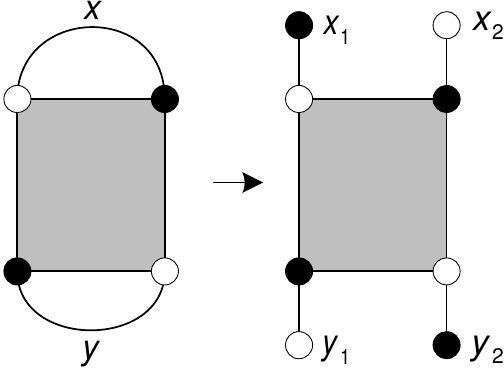}
\hfill
\includegraphics[width=5cm]{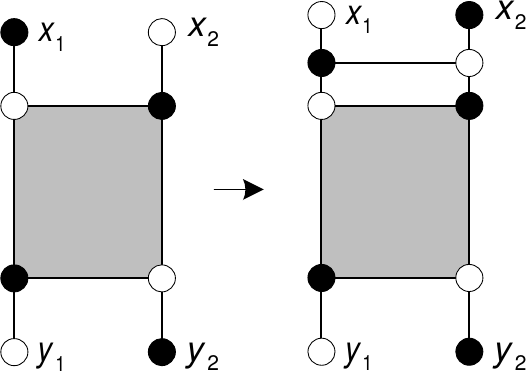}  
\end{center} 
\caption{Creating four-poles.}
\label{fig:Kelmans1and2}
\end{figure}

If $G$ is $(x^+y^-)$-non-Hamiltonian (i.e.~has no Hamiltonian cycle containing $x$ and not $y$), then $G' - \{y_1,y_2\}$ has no Hamiltonian path and is said to be \emph{$(x^+y^-)$-non-traceable}.
If $G$ is $(x^+y^+)$-non-Hamiltonian (i.e.~has no Hamiltonian cycle containing $x$ and $y$), then $G'$ can not be covered by two disjoint paths each starting in $\{x_1,x_2\}$ and ending in $\{y_1,y_2\}$.
$G'$ is said to be \emph{$(x^+y^+)$-non-traceable}.

Note that a four-pole $G'$ inherits its $2$-colouring from the Barnette graph $G$.
If necessary, the colours of the terminals can be altered by extending the four-pole as illustrated in the right side of Figure~\ref{fig:Kelmans1and2}.
It can be checked that the resulting four-pole, while not derived from a Barnette graph, retains the essential properties of $G'$ when used in the constructions that follow.

Since the number of vertices in a four-pole is even, we have the following useful \emph{covering path property}. Suppose $u$ and $v$ are terminals of $G'$ with the same colour.
Then $G' - \{u,v\}$ does not have a Hamiltonian path.

The following theorem is a summary of various results of Kelmans. The main parts of the proof are (1) $\implies$ (2) and (3) $\implies$ (4), which are proved in~\cite{Kelmans}, but with most details left to the reader. The other required implications are trivial except for (4)$\implies$(1) which is claimed by Kelmans in~\cite{Kelmans2003}.

\begin{theorem}[Kelmans]
The following are equivalent:
\begin{enumerate}
\item Every Barnette graph is Hamiltonian (Barnette's Conjecture).
\item Every Barnette graph is $(x^+y^-)$-Hamiltonian for every choice of $x$ and $y$ in the same face.
\item Every cyclically-$4$-edge-connected Barnette graph is Hamiltonian.
\item Every cyclically-$4$-edge-connected Barnette graph is $(x^+y^+)$-Hamiltonian for every choice of $x$ and $y$ in the same face.
\end{enumerate}
\end{theorem}

\begin{proof}
We begin with the easiest implications. (1)$\implies$(3) requires no explanation. (2)$\implies$(4) follows from the following observation. If a Barnette graph $G$ is $(x^+y^+)$-non-Hamiltonian, then there is an edge $z$ adjacent to $y$ and in the same face as $x$, and $G$ is $(x^+z^-)$-non-Hamiltonian. This is because $G$ is cubic, so if a Hamiltonian cycle avoids some edge, it passes through every adjacent edge.

(4)$\implies$(1): The proof is by induction on the number of cyclic-$3$-edge-cuts in $G$. 
If there are none then $G$ is Hamiltonian by (4). 
If there are some then choose one such that one side of the cut contains no cyclic-$3$-edge-cut. 
Let $G_1$ and $G_2$ be the graphs obtained by contracting one side of the cut. 
By a simple degree sum argument, the vertices on each side of the cut share the same colour, so $G_1$ and $G_2$ are bipartite, and hence Barnette graphs. 
Thus both are Hamiltonian by the induction hypothesis.
One of them, say $G_1$, is cyclically-$4$-edge-connected, and hence $(x^+y^+)$-Hamiltonian for all $x$ and $y$. 
Therefore a Hamiltonian cycle can be found in $G_1$ that is compatible with the Hamiltonian cycle in $G_2$, giving a Hamiltonian cycle in $G$.

(1)$\implies$(2): 
Suppose there is an $(x^+y^-)$-non-Hamiltonian Barnette graph $G$. 
Create an $(x^+y^-)$-non-traceable four-pole $G_1'$.
Applying the construction of Figure~\ref{fig:Kelmans1and2}.2 if needed, we may assume that $x_1$ and $y_1$ receive different colours.
Take two copies of $G_1'$ and connect them as shown in Figure~\ref{fig:Kelmans3} to create a new Barnette graph $G_2$.
Using the covering path property and the $(x^+y^-)$-non-traceable property, a straightforward case analysis shows that $G_2$ has no Hamiltonian cycle avoiding the edge marked $e$.
Now take two copies of $G_2$ and connect them as shown in Figure~\ref{fig:Kelmans4} to create the Barnette graph $G_3$, which has no Hamiltonian cycle containing the edge marked $d$. Finally take two copies of $G_3$ and connect them in much the same way as in Figure~\ref{fig:Kelmans4}, but using the edge $d$. 
This yields a non-Hamiltonian Barnette graph.

\begin{figure}[!ht]
\centering
\includegraphics[width=6cm]{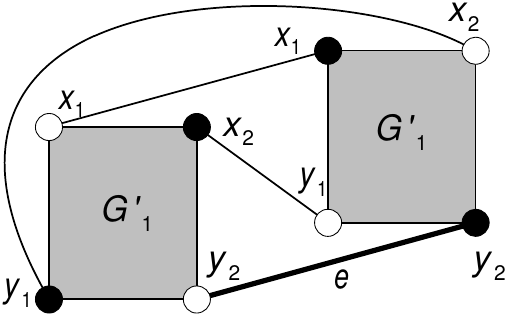}
\caption{Creating $G_2$ in the proof of (1)$\implies$(2).}
\label{fig:Kelmans3}
\end{figure}

\begin{figure}[!ht]
\centering
\includegraphics[width=10cm]{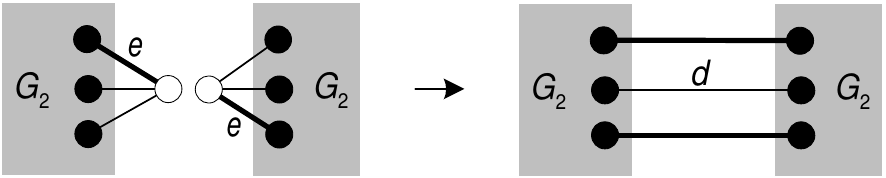}  
\caption{Creating $G_3$ in the proof of (1)$\implies$(2).}
\label{fig:Kelmans4}
\end{figure}

(3)$\implies$(4): 
Suppose there is an $(x^+y^+)$-non-Hamiltonian cyclically-$4$-edge-connected Barnette graph $G$.
Create an $(x^+y^+)$-non-traceable four-pole $G_1'$. 
We may again assume that $x_1$ and $y_1$ get different colours.
As noted at the beginning of the proof, $G$ is also $(x^+z^-)$-non-Hamiltonian for an edge $z$ adjacent to $y$. 
Thus we may also create an $(x^+z^-)$-non-traceable four-pole $G_2'$. 
Take two copies of $G_1'$ and two of $G_2'$ and connect them as shown in Figure~\ref{fig:Kelmans5} to create a cyclically-$4$-edge-connected Barnette graph $G_3$.
Using the covering path and non-traceable properties, a reasonably easy case analysis (based on which of the edges marked $d$ and $e$ are in a supposed Hamiltonian cycle) shows that $G_3$ is non-Hamiltonian.

\begin{figure}[!ht]
\centering
\includegraphics[width=7cm]{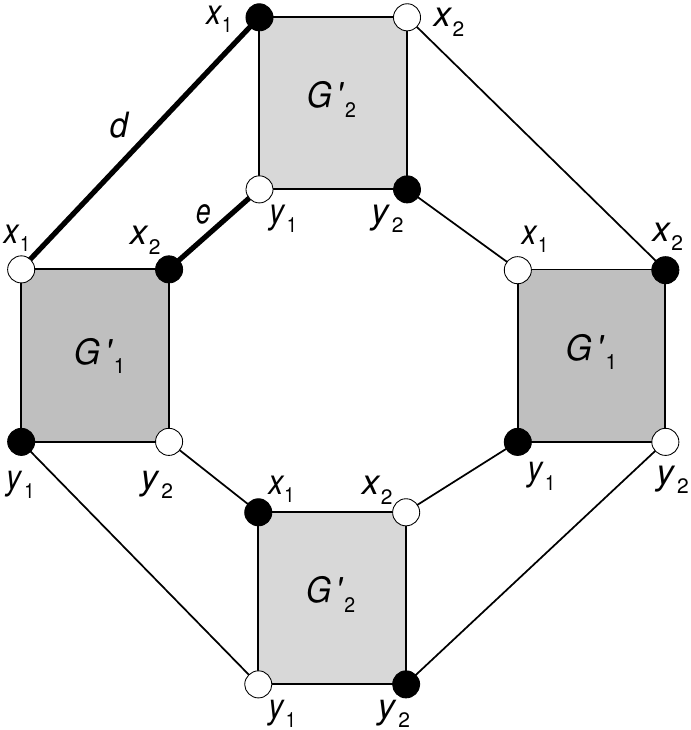}  
\caption{Creating $G_3$ in the proof of (3)$\implies$(4).}
\label{fig:Kelmans5}
\end{figure}
\end{proof}

\bibliographystyle{myNatbibStyle}
\bibliography{paper}

\begin{thebibliography}{22}
\providecommand{\natexlab}[1]{#1}
\providecommand{\url}[1]{\texttt{#1}}
\providecommand{\urlprefix}{}
\expandafter\ifx\csname urlstyle\endcsname\relax
  \providecommand{\doi}[1]{doi:\discretionary{}{}{}#1}\else
  \providecommand{\doi}{doi:\discretionary{}{}{}\begingroup
  \urlstyle{rm}\Url}\fi

\bibitem[{Aldred et~al.(1999)Aldred, Bau, Holton, and McKay}]{Aldred99}
\textsc{Robert E.~L. Aldred, S.~Bau, Derek~A. Holton, and Brendan~D. McKay}.
\newblock Cycles through {$23$} vertices in {$3$}-connected cubic planar
  graphs.
\newblock \emph{Graphs Combin.}, 15(4):373--376, 1999.
\newblock \doi{10.1007/s003730050046}.

\bibitem[{Aldred et~al.(2000)Aldred, Bau, Holton, and McKay}]{Aldred00}
\textsc{Robert E.~L. Aldred, S.~Bau, Derek~A. Holton, and Brendan~D. McKay}.
\newblock Nonhamiltonian 3-connected cubic planar graphs.
\newblock \emph{SIAM J. Discrete Math.}, 13(1):25--32, 2000.
\newblock \doi{10.1137/S0895480198348665}.

\bibitem[{Barnette(1969)}]{Barnette1969a}
\textsc{David~W. Barnette}.
\newblock Conjecture 5.
\newblock In \textsc{William~T. Tutte}, ed., \emph{Recent Progress in
  Combinatorics. Proceedings of the 3rd Waterloo Conference on Combinatorics},
  vol.~3, p. 343. 1969.

\bibitem[{Borowiecki et~al.(2000)Borowiecki, Broere, and
  Mih{\'o}k}]{Borowiecki}
\textsc{Mieczys{\l}aw Borowiecki, Izak Broere, and Peter Mih{\'o}k}.
\newblock Minimal reducible bounds for planar graphs.
\newblock \emph{Discrete Math.}, 212(1-2):19--27, 2000.
\newblock \doi{10.1016/S0012-365X(99)00205-8}.

\bibitem[{Chia and Ong(2002)}]{Chia}
\textsc{Gek~Ling Chia and Siew-Hui Ong}.
\newblock On {B}arnette's conjecture and {CBP} graphs with given number of
  {H}amilton cycles.
\newblock In \emph{Proceedings of the {T}hird {A}sian {M}athematical
  {C}onference (2000)}, pp. 94--111. World Sci. Publ., 2002.
\newblock \doi{10.1142/9789812777461\_0012}.

\bibitem[{Chia and Ong(2007)}]{ChiaOng}
\textsc{Gek~Ling Chia and Siew-Hui Ong}.
\newblock Hamilton cycles in cubic graphs.
\newblock \emph{AKCE Int. J. Graphs Comb.}, 4(3):251--259, 2007.

\bibitem[{de~la Torre(2008)}]{laTorre}
\textsc{Luis de~la Torre}.
\newblock Investigations of {B}arnette's conjecture, 2008.
\newblock
  \urlprefix\url{http://www.math.ucdavis.edu/undergrad/research/thesis/}.
\newblock Bachelor's Thesis, Department of Mathematics, University of
  California, Davis.

\bibitem[{Feder and Subi(2006)}]{Feder}
\textsc{Tomas Feder and Carlos Subi}.
\newblock On {B}arnette's conjecture.
\newblock Tech. Rep. TR06-015, Electronic Colloquium on Computational
  Complexity, 2006.
\newblock \urlprefix\url{http://eccc.hpi-web.de/report/2006/015/}.

\bibitem[{Florek(2010)}]{Florek}
\textsc{Jan Florek}.
\newblock On {B}arnette's conjecture.
\newblock \emph{Discrete Math.}, 310(10-11):1531--1535, 2010.
\newblock \doi{10.1016/j.disc.2010.01.018}.

\bibitem[{Florek(2012)}]{Florek12}
\textsc{Jan Florek}.
\newblock On {B}arnette's conjecture and ${H}^{+-}$ property, 2012.
\newblock \urlprefix\url{http://arxiv.org/abs/1208.4332}.

\bibitem[{Hertel(2005)}]{Hertel}
\textsc{Alexander Hertel}.
\newblock A survey \& strengthening of {B}arnette's conjecture.
\newblock Tech. rep., Dept.\ of Computer Science, University of Toronto, 2005.
\newblock
  \urlprefix\url{http://www.cs.toronto.edu/~ahertel/WebPageFiles/Papers/Streng%
theningBarnette'sConjecture10.pdf}.

\bibitem[{Holton and Aldred(1999)}]{HoltonAldred}
\textsc{Derek Holton and Robert E.~L. Aldred}.
\newblock Planar graphs, regular graphs, bipartite graphs and {H}amiltonicity.
\newblock \emph{Australas. J. Combin.}, 20:111--131, 1999.
\newblock
  \urlprefix\url{http://ajc.maths.uq.edu.au/pdf/20/ocr-ajc-v20-p111.pdf}.

\bibitem[{Holton and Aldred(2000)}]{HoltonAldred00}
\textsc{Derek Holton and Robert E.~L. Aldred}.
\newblock Corrigendum: ``{P}lanar graphs, regular graphs, bipartite graphs and
  {H}amiltonicity'' [{A}ustralas.\ {J}.\ {C}ombin.\ {\bf 20} (1999), 111--131;
  {MR}1723867 (2000i:05117)].
\newblock \emph{Australas. J. Combin.}, 21:311, 2000.
\newblock \urlprefix\url{http://ajc.maths.uq.edu.au/pdf/21/ajc-v21-p311.pdf}.

\bibitem[{Holton et~al.(1985)Holton, Manvel, and McKay}]{Holton}
\textsc{Derek~A. Holton, B.~Manvel, and Brendan~D. McKay}.
\newblock Hamiltonian cycles in cubic {$3$}-connected bipartite planar graphs.
\newblock \emph{J. Combin. Theory Ser. B}, 38(3):279--297, 1985.
\newblock \doi{10.1016/0095-8956(85)90072-3}.

\bibitem[{Kelmans(1994)}]{Kelmans}
\textsc{Alexander~K. Kelmans}.
\newblock \emph{Constructions of cubic bipartite 3-connected graphs without
  {H}amiltonian cycles}, vol. 158 of \emph{American Mathematical Society
  Translations, Series 2}, pp. 127--140.
\newblock American Mathematical Society, 1994.

\bibitem[{Kelmans(2003)}]{Kelmans2003}
\textsc{Alexander~K. Kelmans}.
\newblock On {H}amiltonian cycles in bipartite cubic 3-connected planar graph.
\newblock Tech. Rep. RR26-203, Rutgers University and University of Puerto
  Rico, 2003.
\newblock
  \urlprefix\url{http://rutcor.rutgers.edu/pub/rrr/reports2003/26_2003.ps}.

\bibitem[{Krooss(2004)}]{Krooss}
\textsc{Andr{\'e} Krooss}.
\newblock Die {B}arnette'sche {V}ermutung und die {G}rinberg'sche {F}ormel.
\newblock \emph{An. Univ. Craiova Ser. Mat. Inform.}, 31:59--65, 2004.

\bibitem[{Lu(2011)}]{Lu}
\textsc{Xiaoyun Lu}.
\newblock A note on {B}arnette's conjecture.
\newblock \emph{Discrete Math.}, 311(23-24):2711--2715, 2011.
\newblock \doi{10.1016/j.disc.2011.08.011}.

\bibitem[{Plummer and Pulleyblank(1982)}]{Plummer}
\textsc{Michael~D. Plummer and William~R. Pulleyblank}.
\newblock On proximity to paths and cycles in {$3$}-connected graphs.
\newblock \emph{Ars Combin.}, 14:169--185, 1982.

\bibitem[{Skupie{\'n}(2002)}]{Skupien}
\textsc{Zdzis{\l}aw Skupie{\'n}}.
\newblock Hamiltonicity of planar cubic multigraphs.
\newblock \emph{Discrete Math.}, 251(1-3):163--168, 2002.
\newblock \doi{10.1016/S0012-365X(01)00337-5}.

\bibitem[{Stein(1971)}]{Stein}
\textsc{Sherman~K. Stein}.
\newblock {$B$}-sets and planar maps.
\newblock \emph{Pacific J. Math.}, 37:217--224, 1971.
\newblock \urlprefix\url{http://projecteuclid.org/euclid.pjm/1102970755}.

\bibitem[{Tutte(1946)}]{Tutte}
\textsc{William~T. Tutte}.
\newblock On {H}amiltonian circuits.
\newblock \emph{J. London Math. Soc.}, 21:98--101, 1946.
\newblock \doi{10.1112/jlms/s1-21.2.98}.

\end{thebibliography}
\end{document}